\documentclass[a4paper, 11pt]{amsart}

\usepackage{amsmath,amssymb,microtype,xcolor}
\usepackage[T1]{fontenc}
\usepackage[utf8]{inputenc}
\usepackage[english]{babel}
\usepackage[top=2.7cm,bottom=2.8cm,left=3.0cm,right=3.0cm]{geometry}
\usepackage[bookmarksdepth=2,linktoc=page,colorlinks,linkcolor={red!80!black},citecolor={red!80!black},urlcolor={blue!80!black}]{hyperref}
\usepackage[mathcal]{euscript}                               
\usepackage{mathptmx}                                        

\theoremstyle{plain}
\newtheorem{thm}{Theorem} 

\newtheorem{lemma}[thm]{Lemma}

\newtheorem*{thm*}{Theorem}

\theoremstyle{definition}

\DeclareFontFamily{OT1}{pzc}{}                                
\DeclareFontShape{OT1}{pzc}{m}{it}{<-> s * [1.10] pzcmi7t}{}
\DeclareMathAlphabet{\mathpzc}{OT1}{pzc}{m}{it}
\DeclareSymbolFont{sfoperators}{OT1}{bch}{m}{n} \DeclareSymbolFontAlphabet{\mathsf}{sfoperators} \makeatletter\def\operator@font{\mathgroup\symsfoperators}\makeatother 
\DeclareSymbolFont{cmletters}{OML}{cmm}{m}{it}              
\DeclareSymbolFont{cmsymbols}{OMS}{cmsy}{m}{n}
\DeclareSymbolFont{cmlargesymbols}{OMX}{cmex}{m}{n}
\DeclareMathSymbol{\myjmath}{\mathord}{cmletters}{"7C}     \let\jmath\myjmath 
\DeclareMathSymbol{\myamalg}{\mathbin}{cmsymbols}{"71}     
\DeclareMathSymbol{\mycoprod}{\mathop}{cmlargesymbols}{"60}
\DeclareMathSymbol{\myalpha}{\mathord}{cmletters}{"0B}     \let\alpha\myalpha 
\DeclareMathSymbol{\mybeta}{\mathord}{cmletters}{"0C}      \let\beta\mybeta
\DeclareMathSymbol{\mygamma}{\mathord}{cmletters}{"0D}     \let\gamma\mygamma
\DeclareMathSymbol{\mydelta}{\mathord}{cmletters}{"0E}     \let\delta\mydelta
\DeclareMathSymbol{\myepsilon}{\mathord}{cmletters}{"0F}   \let\epsilon\myepsilon
\DeclareMathSymbol{\myzeta}{\mathord}{cmletters}{"10}      \let\zeta\myzeta
\DeclareMathSymbol{\myeta}{\mathord}{cmletters}{"11}       \let\eta\myeta
\DeclareMathSymbol{\mytheta}{\mathord}{cmletters}{"12}     \let\theta\mytheta
\DeclareMathSymbol{\myiota}{\mathord}{cmletters}{"13}      \let\iota\myiota
\DeclareMathSymbol{\mykappa}{\mathord}{cmletters}{"14}     \let\kappa\mykappa
\DeclareMathSymbol{\mylambda}{\mathord}{cmletters}{"15}    \let\lambda\mylambda
\DeclareMathSymbol{\mymu}{\mathord}{cmletters}{"16}        \let\mu\mymu
\DeclareMathSymbol{\mynu}{\mathord}{cmletters}{"17}        \let\nu\mynu
\DeclareMathSymbol{\myxi}{\mathord}{cmletters}{"18}        \let\xi\myxi
\DeclareMathSymbol{\mypi}{\mathord}{cmletters}{"19}        \let\pi\mypi
\DeclareMathSymbol{\myrho}{\mathord}{cmletters}{"1A}       \let\rho\myrho
\DeclareMathSymbol{\mysigma}{\mathord}{cmletters}{"1B}     \let\sigma\mysigma
\DeclareMathSymbol{\mytau}{\mathord}{cmletters}{"1C}       \let\tau\mytau
\DeclareMathSymbol{\myupsilon}{\mathord}{cmletters}{"1D}   \let\upsilon\myupsilon
\DeclareMathSymbol{\myphi}{\mathord}{cmletters}{"1E}       \let\phi\myphi
\DeclareMathSymbol{\mychi}{\mathord}{cmletters}{"1F}       \let\chi\mychi
\DeclareMathSymbol{\mypsi}{\mathord}{cmletters}{"20}       \let\psi\mypsi
\DeclareMathSymbol{\myomega}{\mathord}{cmletters}{"21}     \let\omega\myomega
\DeclareMathSymbol{\myvarepsilon}{\mathord}{cmletters}{"22}\let\varepsilon\myvarepsilon
\DeclareMathSymbol{\myvartheta}{\mathord}{cmletters}{"23}  \let\vartheta\myvartheta
\DeclareMathSymbol{\myvarpi}{\mathord}{cmletters}{"24}     \let\varpi\myvarpi
\DeclareMathSymbol{\myvarrho}{\mathord}{cmletters}{"25}    \let\varrho\myvarrho
\DeclareMathSymbol{\myvarsigma}{\mathord}{cmletters}{"26}  \let\varsigma\myvarsigma
\DeclareMathSymbol{\myvarphi}{\mathord}{cmletters}{"27}    \let\varphi\myvarphi

\newcommand\N{{\mathbb N}}

\newcommand\R{{\mathbb R}}

\newcommand\T{{\mathbb T}}

\newcommand\hyp{{\textup{hyp}}}

\title{Polynomials over strict semirings do not admit unique factorization}

\author{Alexander Agudelo}
\email{alagudel@impa.br}
\address{Instituto Nacional de Matem\'atica Pura e Aplicada, Rio de Janeiro, Brazil}

\begin{document}

\begin{abstract} 
 In this short note, we prove the claim of the title.
\end{abstract}

\maketitle

\ \vspace{-1cm}


\section*{Introduction}

It is well-known that the polynomial algebra over the tropical semifield does not satisfy unique factorization. For example,
\[
 (T+1)(T^2+1) \ = \ T^3+T^2+T+1 \ = \ (T+1)^3
\]
are two different factorizations of the same polynomial. Note that this failure of unique factorization should not be confused with the unique factorization property for \emph{tropical (polynomial) functions}, which holds since polynomials that define the same function are identified; also cf.\ the comments at the end of section 2 in \cite{Speyer-Sturmfels09}.

It is clear that the above example works for polynomial algebras over any idempotent semiring. What is not so clear, and to our knowledge not mentioned in the literature yet, is that unique factorization fails for polynomial algebras over any strict semiring.

Examples for which unique factorization fails are the polynomial algebra over the next semirings:
\begin{itemize}
	\item Any idempotent semirings.
	\item The natural numbers $\N$.
	\item The non-negative real numbers $\R_{\geq 0}$,
	\item The ambient semiring $(\T^\hyp)^+=\N[\R_{>0}]$ of the tropical hyperfield $\T^\hyp$; cf.\ \cite[Ex.\ 5.7.4]{Lorscheid18} for more details on $\T^\hyp$.
	\item $S_0=\N [x,y]/(xy\sim 1, x+y\sim 1)$. This semiring is interesting since there is a semiring morphism from $S_0$ to $R$ if and only if there is an element $a\in R^{\times}$ such that $a+a^{-1}=1$, and this last property leads to different cases in our proof.
\end{itemize}

\section*{The theorem}

A \emph{(commutative) semiring (with $0$ and $1$} is a set $R$ with two constants $0$ and $1$ and two binary operations $+$ and $\cdot$ such that both $(R,+,0)$ and $(R,\cdot,1)$ are commutative monoids and such that $0\cdot a=0$ and $a(b+c)=ab+ac$ for all $a,b,c\in R$. A semiring $R$ is \emph{strict} if $a+b=0$ implies $a=b=0$ for all $a,b\in R$. It is \emph{without zero divisors} if $ab=0$ implies $a=0$ or $b=0$ for all $a,b\in R$. An element in $R$ is a \emph{unit} if it has an inverse element in the multiplicative monoid of $R$. The set of all units of $R$ will be denoted by $R^\times$. A non-zero non-unit element of $R$ is \emph{irreducible} if it is not a product of two non-units. The semiring $R$ admits the \emph{unique factorization property} if it is without zero divisor and every non-zero non-unit element has a unique factorization into irreducible elements, up to permutations and multiplication by units. The \emph{polynomial algebra} over $R$ is the semiring $R[T]$ of formal sums $\sum_{i\in \N} a_i T^i$, where there is $n$ such that $a_n=0$ for all $i>n$, together with the usual addition and multiplication of polynomials.

Our main theorem is the following.

\begin{thm}\label{thm1}
Polynomial algebras over strict semirings without zero divisors do not satisfy the unique factorization property.
\end{thm}

\section*{The proof}

Let $R$ be a strict semiring without zero divisors, we know that $R[T]$ is without zero divisor and the units of $R[T]$ are the units of $R$ (cf.\ \cite[Ex.\ 4.23]{Golan99}). This means, as usual, two irreducible polynomials of different degree can not differ by a unit. Thus if we find one polynomial that admits two different factorization into irreducible polynomials of different degree we are done.

Note that the polynomial $T^5+T^4+T^3+T^2+T+1$ admits the following two factorization
	\[(T+1)(T^4+T^2+1) = (T^3+1)(T^2+T+1).
\]

The next lemmas tell us when the polynomial above are irreducible.

\begin{lemma}\label{lemma1}
For all $a\in R^\times$ and $n\geq 1$ the polynomial $T^n+a$ is irreducible. 
\end{lemma}

\begin{proof}
It is obvious that $T+a$ is irreducible. For $n\geq 2$ assume that $T^n+a$ is reducible, then there are $1\leq m < n$ and $c_0, \ldots, c_{m-1}, d_0, \ldots, d_{n-m-1}\in R$ such that
	\[T^n+a = (T^m + c_{m-1} T^{m-1} + \cdots + c_0)(T^{n-m} +d_{n-m-1} T^{n-m-1} + \cdots + d_0).
\]
Therefore $c_0 d_0=a$ and $\sum_{i=1}^k c_i d_{k-i}=0$ for all $1\leq k\leq n-1$. In particular, $d_0$ is a unit. For $k=m$, we get $d_0+\sum_{i=1}^{m-1} c_i d_{k-i}=0$. Since $R$ is strict, $d_0=0$, which is a contradiction.
\end{proof}

\begin{lemma}\label{lemma2}
Either $T^2+T+1$ is irreducible or there is $a\in R^\times$ such that $a+a^{-1}=1$. In the latter case, $T^2+T+1=(T+a)(T+a^{-1})$.
\end{lemma}

\begin{proof}
Suppose that $T^2+T+1$ is reducible. Then there are $c,d\in R$ such that
	\[T^2+T+1 = (T+c)(T+d).
\]
This implies $cd=1$ and $c+d=1$.
\end{proof}

\begin{lemma} \label{lemma3}
Either $T^4+T^2+1$ is irreducible or there is $a\in R^\times$ such that $a+a^{-1}=1$. In the latter case, $T^4+T^2+1=(T^2+a)(T^2+a^{-1})$.
\end{lemma}

\begin{proof}
Suppose that $T^4+T^2+1$ is reducible. Now we have two possibilities, $T^4+T^2+1$ factors as the product of one linear and one cubic polynomial or as the product of two quadratics polynomials.

In the first case, there are $c_0,d_0,d_1,d_2\in R$ such that
	\[T^4+T^2+1 = (T+c_0)(T^3+d_2T^2+d_1T+d_0).
\]
Therefore $c_0d_0=1$ and $c_0+d_2=0$, which is impossible since $R$ is strict.

In the second case, there are $c_0,c_1,d_0,d_1\in R$ such that 
	\[T^4+T^2+1 = (T^2+c_1T+c_0)(T^2+d_1T+d_0).
\]
Therefore $c_0d_0=1$, $c_0d_1+c_1d_0=0$, $c_1d_1+c_0+d_0=1$ and $c_1+d_1=0$. Since $R$ is strict, we get $c_1=d_1=0$, $c_0d_0=1$ and $c_0+d_0=1$.
\end{proof}

We are ready to finish the proof of Theorem \ref{thm1}.

If there is no $a\in R^\times$ such that $a+a^{-1}=1$, then Lemmas \ref{lemma1}, \ref{lemma2} and \ref{lemma3} show that all factors in
	\[(T+1)(T^4+T^2+1) = (T^3+1)(T^2+T+1)
\]
are irreducible. Otherwise, we have
	\[(T+1)(T^2+a)(T^2+a^{-1}) = (T^3+1)(T+a)(T+a^{-1}).
\]
Lemma \ref{lemma1} implies that all factors in the above expression are irreducible. In both cases we have found two different factorization of $T^5+T^4+T^3+T^2+T+1$ into irreducible polynomials of different degree. \hfill\qed

\begin{small}
 \bibliographystyle{plain}
 \bibliography{biblio}
\end{small}

\end{document}